\documentclass[a4paper]{article}
\usepackage{latexsym}
\usepackage{amsmath}
\usepackage{amssymb}
\usepackage{amsbsy}
\usepackage{amsthm}
\usepackage{amsfonts}
\usepackage{mathrsfs}
\usepackage{bm}
\usepackage{graphicx}
\usepackage{color,xcolor}
\usepackage{amscd}
\usepackage[colorlinks,linkcolor=red,anchorcolor=blue,citecolor=green]{hyperref}
\usepackage{cleveref}

\usepackage{url}
\urldef{\mailsa}\path|{alfred.hofmann,ursula.barth,ingrid.beyer,natalie.brecht,|
\urldef{\mailsb}\path|christine.guenther,frank.holzwarth,piamaria.karbach,|
\urldef{\mailsc}\path|anna.kramer,erika.siebert-cole,lncs}@springer.com|

\title{Factorization of Scalar Piecewise Continuous Almost Periodic Functions}

\author{Liangping Qi\thanks{Chern Institute of Mathematics, Nankai University, Tianjin, PR China. Email: lqi@nankai.edu.cn.}, Rong Yuan\thanks{School of Mathematical Sciences, Beijing Normal University, Beijing, PR China. Email: ryuan@bnu.edu.cn.}}
\date{}

\DeclareMathOperator{\sgn}{sgn}

\newtheorem{theorem}{Theorem}[section]
\newtheorem{definition}[theorem]{Definition}
\newtheorem{lemma}[theorem]{Lemma}

\newtheorem{remark}[theorem]{Remark}

\begin{document}

\maketitle

\begin{abstract}
  This paper is to characterize piecewise continuous almost periodic functions as the product of Bohr almost periodic functions and sequences. As an application, the result is used to discuss piecewise continuous almost periodic solutions of impulsive differential equations.
\end{abstract}



\section{Introduction}\label{sec:intr}

Impulsive differential equations can model many real processes and phenomena studied in mechanics, radio engineering, communication security, control theory, neural networks, etc. since the states of many evolutionary processes are often subject to instantaneous perturbations and experience abrupt changes at certain moments of time. Piecewise continuous almost periodic (p.c.a.p., for short) functions come from the study of impulsive differential equations and certain problems in real phenomena, see e.g. \cite{ASS10,BS93,HL89,LBS89,SP95,Sta12,Tka16} and \cite{Akh08,BMT08,N'G01,OT06,Zha03} respectively for various works devoted to impulsive differential equations and almost periodic differential equations.

The definition of p.c.a.p. functions is complex and consists of three conditions (\Cref{df:pcap}). The aim of this paper is to present a property of p.c.a.p. functions. This property can be used as not only a good characterization of p.c.a.p. functions, but also a powerful tool in the discussion of p.c.a.p. solutions to impulsive differential equations. The main result is formulated as follows. Notations and terminologies will be explained later.
\begin{theorem}\label{th:factor1}
Suppose that $h\in PCAP(\mathbb{R}, \mathbb{R})$ has discontinuities at the points of a subset of a Wexler sequence $\{\tau_j\}_{j\in \mathbb{Z}}$ and
\begin{equation}\label{eq:factor-h}
  \inf_{n\in \mathbb{Z}} |h(\tau_n)|>0, \quad \inf_{n\in \mathbb{Z}} |h(\tau_n^+)|>0.
\end{equation}
If the difference equation
\begin{equation}\label{eq:factor-c}
  x(n)=\frac{h(\tau_n^+)}{h(\tau_n)}\cdot x(n-1)
\end{equation}
has a bounded solution $v$ with $\inf_{n\in \mathbb{Z}} |v(n)|>0$, then there exists a unique $f\in AP(\mathbb{R}, \mathbb{R})$ and $u\in AP(\mathbb{Z}, \mathbb{R})$, up to a nonzero multiplicative constant, such that
\begin{equation}\label{eq:factor}
  h(t)=f(t)\cdot u(n), \quad \tau_n<t\leq \tau_{n+1}, n\in \mathbb{Z}.
\end{equation}
\end{theorem}

\begin{theorem}\label{th:factor2}
Suppose that $f\in AP(\mathbb{R}, \mathbb{R})$, $u\in AP(\mathbb{Z}, \mathbb{R})$ and $\{\tau_j\}_{j\in \mathbb{Z}}$ is a Wexler sequence. Then the function $h$ defined by \eqref{eq:factor} is p.c.a.p.
\end{theorem}

\Cref{th:factor1,th:factor2} give us a good understanding of p.c.a.p. functions by characterizing them as the synthesis of Bohr almost periodic functions and sequences. So it is convenient to generate new p.c.a.p. functions and investigate their properties from Bohr functions and sequences. In addition, \eqref{eq:factor} shows that p.c.a.p. functions have a delicate structure which deserves more study.

As for impulsive differential equations, \Cref{th:factor1,th:factor2} indicate the direction of looking for solutions in the factorization form and provide an answer to the almost periodicity of the functions in an assumption (see (A1) below) of \cite{LY13,LW14,TLWL14,TCQ15,ZLJZ14,ZLY12}. Changing variables is an important tool in dynamics. To simplify the analysis of the following impulsive dynamical system in $\mathbb{R}^d$,
\begin{equation}\label{eq:IDE1}
\begin{cases}
&y_i'=g_i(y_1, \ldots, y_d, t), \quad t\neq \tau_n,\\
&y_i(\tau_n^+)=[1+b_i(n)]y_i(\tau_n), \quad n\in \mathbb{Z}_+, i=1, \ldots, d,
\end{cases}
\end{equation}
where $\{g_i\}_{i=1}^d$ are continuous functions and $0<\tau_1<\tau_2<\cdots$, it is convenient to make the following change of variables
\begin{equation}\label{eq:change1}
\begin{cases}
  &y_i(t)\mapsto \frac{y_i(t)}{\varphi_i(t)}, \quad t>0,\\
  &\varphi_i(t):=\prod_{0<\tau_k<t}[1+b_i(t)], \quad i=1, \ldots, d,
\end{cases}
\end{equation}
which transforms \eqref{eq:IDE1} to the following system
\begin{equation}\label{eq:DE1}
  y_i'=\frac{g_i(\varphi_1(t)y_1, \ldots, \varphi_d(t)y_d, t)}{\varphi_i(t)}, \quad t\neq \tau_n, n\in \mathbb{Z}_+, i=1, \ldots, d,
\end{equation}
so that $(\phi_1, \ldots, \phi_d)$ is a solution to \eqref{eq:IDE1} if and only if $(\phi_1/\varphi_1, \ldots, \phi_d/\varphi_d)$ is a solution to \eqref{eq:DE1}. In order to obtain p.c.a.p. solutions to \eqref{eq:IDE1}, \cite{LY13,LW14,TLWL14,TCQ15,ZLJZ14,ZLY12} make the assumption
\begin{description}
  \item[(A1)] the functions $\{\varphi_i\}_{i=1}^d$ are bounded and p.c.a.p.
\end{description}
However, there exists no theoretical reason for the almost periodicity of the functions in \eqref{eq:change1} in literature. Our \Cref{th:factor1,th:factor2} can solve this problem. To show this, we first need a modification of the functions in \eqref{eq:change1} since they are only defined for $t>0$. For $i=1$, $\ldots$, $d$, let
\begin{equation}\label{eq:seq}
u_i(n):=
\begin{cases}
&\prod_{k=1}^n [1+b_i(k)], \quad n>0,\\
&1, \quad n=1,\\
&\prod_{k=n+1}^0 \frac{1}{1+b_i(k)}, \quad n<0,
\end{cases}
\end{equation}
and
\begin{equation}\label{eq:fct}
\begin{split}
\psi_i(t):&=
\begin{cases}
&\prod_{\tau_0<\tau_k<t} [1+b_i(k)], \quad t\geq \tau_1,\\
&1, \quad \tau_0<t\leq \tau_1,\\
&\prod_{t\leq \tau_k\leq \tau_0} \frac{1}{1+b_i(k)}, \quad t\leq \tau_0,
\end{cases}\\
&=u_i(n), \quad \tau_n<t\leq \tau_{n+1}, n\in \mathbb{Z},
\end{split}
\end{equation}
then $u_i$ is a solution to
\begin{equation}\label{eq:CE1}
  x(n)=[1+b_i(n)]x(n-1), \quad n\in \mathbb{Z},
\end{equation}
and
\begin{equation*}
  \psi_i(\tau_n^+)=[1+b(n)]\psi_i(\tau_n), \quad n\in \mathbb{Z}.
\end{equation*}
Instead of (A1), we propose the assumption
\begin{description}
  \item[(A2)] for $i=1$, $\ldots$, $d$, the sequence $b_i$ is almost periodic, $\inf_{n\in \mathbb{Z}} |1+b_i(n)|>0$, and \eqref{eq:CE1} admits a bounded solution $v_i$ with $\inf_{n\in \mathbb{Z}} |v_i(n)|>0$.
\end{description}
Although (A1) looks simpler than (A2), (A1) is fruition rather than source and has no theoretical basis. (A2) concerns the essence and will be shown crucial in \Cref{th:hc-ap}. Our answer to the almost periodicity of the functions in the transformation
\begin{equation}\label{eq:change2}
  y_i(t)\mapsto \frac{y_i(t)}{\psi_i(t)}, \quad t\in \mathbb{R}, i=1, \ldots, d
\end{equation}
associated with impulsive differential equations reads as follows.
\begin{theorem}\label{th:change}
Suppose that $\{\tau_j\}_{j\in \mathbb{Z}}$ is a Wexler sequence and (A2) holds.
\begin{description}
  \item[(i)] If $(\phi_1, \ldots, \phi_d)$ is a p.c.a.p. solution to
\begin{equation}\label{eq:IDE2}
\begin{cases}
&y_i'=g_i(y_1, \ldots, y_d, t), \quad t\neq \tau_n,\\
&y_i(\tau_n^+)=[1+b_i(n)]y_i(\tau_n), \quad n\in \mathbb{Z}, i=1, \ldots, d,
\end{cases}
\end{equation}
  and $\phi_i(\tau_n)\neq 0$ for all $n\in \mathbb{Z}$ and $i=1$, $\ldots$, $d$, then $(\phi_1/\psi_1, \ldots, \phi_d/\psi_d)$ is a Bohr almost periodic solution to
\begin{equation}\label{eq:DE2}
  y_i'=\frac{g_i(\psi_1(t)y_1, \ldots, \psi_d(t)y_d, t)}{\psi_i(t)}, \quad t\neq \tau_n, n\in \mathbb{Z}, i=1, \ldots, d.
\end{equation}
  \item[(ii)] If $(\phi_1/\psi_1, \ldots, \phi_d/\psi_d)$ is a Bohr almost periodic solution to \eqref{eq:DE2}, then\\
      $(\phi_1, \ldots, \phi_d)$ is a p.c.a.p. solution to \eqref{eq:IDE2}.
\end{description}
\end{theorem}

This paper is organized as follows. In \Cref{sec:pre} we prepare some prerequisite knowledge. In \Cref{sec:hc} we study almost periodic solutions to homogeneous difference equations, which is crucial in proving our main result. In \Cref{sec:factor} we prove \Cref{th:factor1,th:factor2,th:change}.

\section{Preliminaries}\label{sec:pre}

Let $\mathbb{G}=\mathbb{R}$ or $\mathbb{Z}$. A two-sided real sequence is a function $\{u_n\}_{n\in \mathbb{Z}}=\{u(n)\}_{n\in \mathbb{Z}}$ from $\mathbb{Z}$ to $\mathbb{R}$. In the following both notations will be used.

\begin{definition}[{\cite[p. 45]{Cor89}, \cite[p. 183]{SP95}}]\label{df:ap}
A continuous function $f: \mathbb{G}\rightarrow \mathbb{R}$ is called Bohr almost periodic if given any $\epsilon>0$, the $\epsilon$-almost periodic set of $f$,
\begin{equation*}
    T(f, \epsilon):=\{\tau\in\mathbb{G};\ |f(t+\tau)-f(t)|<\epsilon, \forall t\in\mathbb{G}\}
\end{equation*}
is relatively dense, that is, there is a positive number $l=l(\epsilon)\in \mathbb{G}$ satisfying $[a, a+l]\cap T(f, \epsilon)\neq \emptyset$ for all $a\in \mathbb{G}$.
\end{definition}

Denote by $AP(\mathbb{G}, \mathbb{R})$ the set of all Bohr almost periodic functions from $\mathbb{G}$ to $\mathbb{R}$. $AP(\mathbb{G}, \mathbb{R})$ is a Banach space equipped with the norm $\|f\|=\sup_{t\in \mathbb{G}}|f(t)|$.

We call a sequence $\{\tau_j\}_{j\in\mathbb{Z}}\subset\mathbb{R}$ admissible if $\tau_j<\tau_{j+1}$ for all $j\in\mathbb{Z}$ and $\lim_{j\rightarrow\pm\infty}\tau_j=\pm\infty$. Put $\tau_j^k=\tau_{j+k}-\tau_j$, where $j$, $k\in\mathbb{Z}$.

\begin{definition}[{\cite[p. 195]{SP95}}]\label{df:Wexler}
An admissible sequence $\{\tau_j\}_{j\in\mathbb{Z}}$ is called a Wexler sequence if $\inf_{j\in\mathbb{Z}}\tau_j^1>0$ and the family of sequences
\begin{equation*}
  \{\{\tau_j^k\}\}:=\{\{\tau_j^k\}_{j\in\mathbb{Z}}\}_{k\in\mathbb{Z}}
\end{equation*}
is equi-potentially almost periodic, i.e. for each $\epsilon>0$ the common $\epsilon$-almost periodic set of all the sequences $\{\{\tau_j^k\}\}$,
\begin{equation*}
  T(\{\{\tau_j^k\}\}, \epsilon):=\big\{p\in\mathbb{Z}; |\tau_{j+p}^k-\tau_j^k|<\epsilon\text{ for all }j, k\in\mathbb{Z}\big\}
\end{equation*}
is relatively dense.
\end{definition}

Let $PC(\mathbb{R}, \mathbb{R})$ be the set of piecewise continuous functions $h:\mathbb{R}\rightarrow \mathbb{R}$ which have discontinuities of the first kind only at the points of a subset of an admissible sequence $\{\tau_j=\tau_j(h)\}_{j\in\mathbb{Z}}$ and are continuous from the left at $\{\tau_j\}_{j\in\mathbb{Z}}$. Note that we have included continuous functions in $PC(\mathbb{R}, \mathbb{R})$ since the empty set is a subset of all admissible sequences.

\begin{definition}[{\cite[p. 201]{SP95}}]\label{df:pcap}
A function $h\in PC(\mathbb{R}, \mathbb{R})$ is called piecewise continuous almost periodic (p.c.a.p.) if the following conditions hold:
\begin{description}
  \item[(i)] There is an admissible sequence $\{\tau_j\}_{j\in\mathbb{Z}}$ containing possible discontinuities of $h$. The family of sequences $\{\{\tau_j^k\}\}$ is equi-potentially almost periodic.
  \item[(ii)] For each $\epsilon>0$ there exists a $\delta=\delta(\epsilon)>0$ such that $|h(s)-h(t)|<\epsilon$ whenever $s$, $t\in(\tau_j, \tau_{j+1}]$ for some $j\in\mathbb{Z}$ and $|s-t|<\delta$.
  \item[(iii)] For each $\epsilon>0$, the $\epsilon$-almost periodic set of $h$,
\begin{equation*}
\begin{split}
    T(h, \epsilon):=\{&\tau\in\mathbb{R}; |h(t+\tau)-h(t)|<\epsilon \text{ for all } t\in\mathbb{R}\\
    &\text{such that }|t-\tau_j|>\epsilon, j\in\mathbb{Z}\}
\end{split}
\end{equation*}
is relatively dense.
\end{description}
\end{definition}
Denote by $PCAP(\mathbb{R}, \mathbb{R})$ the set of all p.c.a.p. functions. Elementary properties of p.c.a.p. functions are listed below.
\begin{lemma}[{\cite[p. 206, 214]{SP95}}]\label{lem:pcap}
Suppose that $h$, $h_1\in PCAP(\mathbb{R}, \mathbb{R})$ have discontinuities at the points of a subset of the same Wexler sequence $\{\tau_j\}_{j\in\mathbb{Z}}$, then \emph{(i)} $hh_1\in PCAP(\mathbb{R}, \mathbb{R})$; \emph{(ii)} $h/h_1\in PCAP(\mathbb{R}, \mathbb{R})$ if $\inf_{t\in \mathbb{R}} |h_1(t)|>0$; \emph{(iii)} $\{h(\tau_n)\}_{n\in \mathbb{Z}}\in AP(\mathbb{Z}, \mathbb{R})$.
\end{lemma}

\section{Homogeneous difference equations}\label{sec:hc}

To prove \Cref{th:factor1,th:factor2} we need some results on difference equations. We shall give an equivalent condition on the existence of almost periodic solutions to the following difference equation in $\mathbb{R}$,
\begin{equation}\label{eq:hc}
  x(n)=a(n)x(n-1), \quad n\in \mathbb{Z},
\end{equation}
where $a(n)\neq 0$. It is easy to check that for each $x_0\in \mathbb{R}$ the initial value problem
\begin{equation}\label{eq:hc-ivp}
\begin{cases}
x(n)=a(n)x(n-1), \quad n\in \mathbb{Z},\\
x(n_0)=x_0
\end{cases}
\end{equation}
admits on $\mathbb{Z}$ a unique solution $\phi$ given by
\begin{equation}\label{eq:hc-s}
\phi(n)=
\begin{cases}
\big[\prod_{k=n_0+1}^na(k)\big]\cdot x_0, \quad n>n_0,\\
x_0, \quad n=n_0,\\
\big[\prod_{k=n+1}^{n_0}\frac{1}{a(k)}\big]\cdot x_0, \quad n<n_0,
\end{cases}
\end{equation}
which is trivial (identically zero) if and only if it attains zero value at any point. Define a function, which is essentially the Cauchy matrix of \eqref{eq:hc}, by
\begin{equation}\label{eq:hc-V}
V(n, m)=
\begin{cases}
\prod_{k=m+1}^na(k), \quad n>m,\\
1, \quad n=m,\\
\prod_{k=n+1}^m\frac{1}{a(k)}, \quad n<m,
\end{cases}
\end{equation}
where $m$, $n\in \mathbb{Z}$. It is easy to see that
\begin{align*}
V(n, m)&=a(n)\cdot V(n-1, m)=V(n, m-1)\cdot \frac{1}{a(m)}\\
&=\frac{1}{V(m, n)}=V(n, l)V(l, m)
\end{align*}
for all $l$, $m$, $n\in \mathbb{Z}$. Therefore, $\phi(n)=V(n, n_0)\cdot x_0$.

\begin{lemma}\label{lem:hc-s}
Suppose that $u$ and $v$ are two nontrivial (not identically zero) solutions to \eqref{eq:hc}, then there exists a constant $c\in \mathbb{R}$, $c\neq 0$ such that $u=cv$.
\end{lemma}

We are interested in almost periodic solutions to \eqref{eq:hc}. The following lemma provides a necessary condition.
\begin{lemma}\label{lem:1}
Suppose that $a\in AP(\mathbb{Z}, \mathbb{R})$ and $0<\vartheta:=\inf_{n\in \mathbb{Z}}|a(n)|<\|a\|$. If \eqref{eq:hc} has an almost periodic solution, then $\vartheta<1<\|a\|$.
\end{lemma}
\begin{proof}
Assume the contrary that $\|a\|\leq 1$. It follows that
\begin{equation*}
  |a(n)|\leq 1, \frac{1}{|a(n)|}\geq 1, \quad n\in \mathbb{Z}.
\end{equation*}
If $u\in AP(\mathbb{Z}, \mathbb{R})$ is a solution to \eqref{eq:hc}, then
\begin{align*}
&|u(n)|\leq |u(0)|, \quad n\in \mathbb{Z}_+,\\
&|u(n)|\geq |u(0)|, \quad n\in \mathbb{Z}_-.
\end{align*}
Since $\vartheta<\|a\|$, \eqref{eq:hc} admits no constant solution. Therefore, $u$ is not constant on $\mathbb{Z}$, and there is an $l\in \mathbb{Z}_+$ such that $|u(l)|<|u(0)|$. Let $\epsilon$ be a number satisfying $0<\epsilon<|u(0)|-|u(l)|$. The set $T(u, \epsilon)$ is relatively dense from the almost periodicity of $u$. Let $p\in T(u, \epsilon)$, $p+l<0$. A direct calculation implies the contradiction
\begin{equation*}
  |u(0)|\leq |u(p+l)|<|u(l)|+\epsilon<|u(0)|.
\end{equation*}

The proof for the case of $\vartheta\geq 1$ is similar. So we omit it.
\end{proof}

The following lemma rules out the possibility that $\inf_{n\in \mathbb{Z}} |u(n)|=0$ for a desired solution $u$.
\begin{lemma}\label{lem:notAA}
Suppose that $a\in AP(\mathbb{Z}, \mathbb{R})$, $\inf_{n\in \mathbb{Z}} |a(n)|=\vartheta>0$, and $u$ is a nontrivial bounded solution to \eqref{eq:hc}, $\inf_{n\in \mathbb{Z}} |u(n)|=0$. Then $u\notin AA(\mathbb{Z}, \mathbb{R})$, the space of real almost automorphic sequences $v$ such that from any sequence $\{n_k'\}_{k=1}^\infty\subset \mathbb{Z}$ we can extract a subsequence $\{n_k\}_{k=1}^\infty\subset \{n_k'\}_{k=1}^\infty$ with
\begin{equation*}
  \lim_{k\rightarrow \infty} v(n+n_k)=w(n), \lim_{k\rightarrow \infty} w(n-n_k)=v(n), \quad n\in \mathbb{Z},
\end{equation*}
for some sequence $w$.
\end{lemma}
\begin{proof}
Since $u$ is nontrivial, $u(n)\neq 0$ for all $n\in \mathbb{Z}$. Let $\{n_k'\}_{k=1}^\infty\subset \mathbb{Z}$ be a sequence such that $\lim_{k\rightarrow \infty} u(n_k')=0$. Using the diagonal procedure, it follows from the boundedness of $u$ that there exists a subsequence $\{n_k\}_{k=1}^\infty\subset \{n_k'\}_{k=1}^\infty$ with $\lim_{k\rightarrow \infty} u(n+n_k)=\hat{u}(n)$ for all $n\in \mathbb{Z}$ and some sequence $\hat{u}$. In particular, $\hat{u}(0)=0$ by assumption. Because $a\in AP(\mathbb{Z}, \mathbb{R})$, we may assume further that $\lim_{k\rightarrow \infty} a(n+n_k)=\hat{a}(n)$ uniformly for all $n\in \mathbb{Z}$ and some $\hat{a}\in AP(\mathbb{Z}, \mathbb{R})$. Consequently, the equality
\begin{equation*}
  u(n+n_k)=a(n+n_k)u(n+n_k-1)
\end{equation*}
implies
\begin{equation*}
  \hat{u}(n)=\hat{a}(n)\hat{u}(n-1)
\end{equation*}
for all $n\in \mathbb{Z}$. Moreover, from $\vartheta\leq |a(n+n_k)|\leq \|a\|$ it follows that $\vartheta\leq |\hat{a}(n)|\leq \|a\|$ for all $n\in \mathbb{Z}$. So, $\hat{u}(n)=0$ for all $n\in \mathbb{Z}$. Hence $\lim_{k\rightarrow \infty} \hat{u}(n-n_k)=0\neq u(n)$ for all $n\in \mathbb{Z}$. Thus $u\notin AA(\mathbb{Z}, \mathbb{R})$.
\end{proof}

\begin{remark}\label{re:notAA}
Lemma 4.7 in \cite{Yua10} proves that an almost periodic solution $\phi$ to $x(n+1)=C(n)x(n)$, $n\in \mathbb{Z}$, where $C\in AP(\mathbb{Z}, \mathbb{R}^d)$, $d\in \mathbb{Z}_+$, satisfies either $\inf_{n\in \mathbb{Z}} |\phi(n)|>0$ or $\phi(n)=0$ for all $n\in \mathbb{Z}$. Our lemma shows that such a nontrivial bounded solution $u$ with $\inf_{n\in \mathbb{Z}} |u(n)|=0$ is even not almost automorphic.
\end{remark}

Consider auxiliary equations
\begin{align}
x(n)&=|a(n)|\cdot x(n-1),\label{eq:hc+}\\
x(n)&=[\sgn a(n)]\cdot x(n-1),\label{eq:hcsgn}
\end{align}
of which the Cauchy matrices are respectively $|V(n, m)|$ and $\sgn V(n, m)$, where $m$, $n\in \mathbb{Z}$. It is easy to prove the following result.

\begin{lemma}\label{lem:hc-aux}
The following statements are true.
\begin{description}
  \item[(i)] If $u$ is a solution to \eqref{eq:hc}, then $\{|u(n)|\}_{n\in \mathbb{Z}}$ is a solution to \eqref{eq:hc+}.
  \item[(ii)] If $v$ is a solution to \eqref{eq:hc+}, then $\{v(n)\cdot \sgn V(n, m)\}_{n\in \mathbb{Z}}$ is a solution to \eqref{eq:hc} for each $m\in \mathbb{Z}$.
  \item[(iii)] If $u$, $v$, $w$ are respectively nontrivial solutions to \eqref{eq:hc}, \eqref{eq:hc+}, \eqref{eq:hcsgn}, then $v/w=v\cdot w$, $u/w=u\cdot w$, $u/v=v/u$ are respectively solutions to \eqref{eq:hc}, \eqref{eq:hc+}, \eqref{eq:hcsgn}.
  \item[(iv)] $u$ is a solution, bounded or with $\inf_{n\in \mathbb{Z}} |u(n)|>0$, to \eqref{eq:hc}, if and only if $v$ is a solution, bounded or with $\inf_{n\in \mathbb{Z}} |v(n)|>0$, to \eqref{eq:hc+}, where $v(n)=|u(n)|$, $u(n)=v(n)\cdot [\sgn V(n, m)]\cdot u(m)/|u(m)|=v(n)\cdot [\sgn V(n, m)u(m)]$.
  \item[(v)] If $u$ is an almost periodic solution to \eqref{eq:hc}, then $\{|u(n)|\}_{n\in \mathbb{Z}}$ is an almost periodic solution to \eqref{eq:hc+}. If in addition, $\inf_{n\in \mathbb{Z}} |u(n)|>0$, then $\{\sgn u(n)\}_{n\in \mathbb{Z}}$ is an almost periodic solution to \eqref{eq:hcsgn}.
\end{description}
\end{lemma}

One may expect integer-valued almost periodic sequences to be periodic.
\begin{lemma}\label{lem:hcsgn-s}
Suppose that $a\in AP(\mathbb{Z}, \mathbb{R})$ and $\inf_{n\in \mathbb{Z}}|a(n)|=\vartheta>0$, then all solutions to \eqref{eq:hcsgn} are $|2p|$-periodic, where $p\in T(a, \epsilon)$, $p\neq 0$ and $0<\epsilon<2\vartheta$.
\end{lemma}
\begin{proof}
It suffices to consider nontrivial solutions $u$ of \eqref{eq:hcsgn}. Let $\epsilon$, $0<\epsilon<2\vartheta$, be given. For each $p\in T(a, \epsilon)$ and $n\in \mathbb{Z}$, the inequality
\begin{equation*}
  |a(n+p)-a(n)|<\epsilon<2\vartheta
\end{equation*}
implies $\sgn a(n+p)=\sgn a(n)$. Consequently, from the equalities
\begin{equation*}
V(n+2p, n)=
\begin{cases}
\prod_{k=n+1}^{n+2p} a(k)=\prod_{k=n+1}^{n+p} a(k)a(k+p), \quad p>0,\\
1, \qquad p=0,\\
\prod_{k=n}^{n+2p+1} \frac{1}{a(k)}=\prod_{k=n}^{n+p+1} \frac{1}{a(k)a(k+p)}, \quad p>0,
\end{cases}
\end{equation*}
it follows that $\sgn V(n+2p, n)=1$ for all $n\in \mathbb{Z}$. Therefore,
\begin{equation*}
  u(n+2p)=[\sgn V(n+2p, n)]\cdot u(n)=u(n)
\end{equation*}
for all $n\in \mathbb{Z}$.
\end{proof}

We are in the position giving an equivalent condition for the existence of almost periodic solutions to \eqref{eq:hc}.
\begin{theorem}\label{th:hc-ap}
Suppose that $a\in AP(\mathbb{Z}, \mathbb{R})$ and $\inf_{n\in \mathbb{Z}}|a(n)|=\vartheta>0$, then a nontrivial solution $u$ to \eqref{eq:hc} is almost periodic if and only if $u$ is bounded on $\mathbb{Z}$ and $\inf_{n\in \mathbb{Z}} |u(n)|>0$.
\end{theorem}
\begin{proof}
Necessity. Suppose that $u\in AP(\mathbb{Z}, \mathbb{R})$ is a nontrivial solution to \eqref{eq:hc}. Clear, $u$ is bounded on $\mathbb{Z}$. From \Cref{lem:notAA} it follows that $\inf_{n\in \mathbb{Z}} |u(n)|>0$.

Sufficiency. Let $u$ be a bounded solution to \eqref{eq:hc} such that $\inf_{n\in \mathbb{Z}} |u(n)|>0$. Then $\{|u(n)|\}_{n\in \mathbb{Z}}$ and $\{\ln |u(n)|\}_{n\in \mathbb{Z}}$ are respectively bounded solutions to \eqref{eq:hc+} and the auxiliary equation
\begin{equation}\label{eq:hcln}
  x(n)-x(n-1)=\ln |a(n)|, \quad n\in \mathbb{Z}.
\end{equation}
Because $\{\ln |a(n)|\}_{n\in \mathbb{Z}}\in AP(\mathbb{Z}, \mathbb{R})$, $\{\ln |u(n)|\}_{n\in \mathbb{Z}}$ is an almost periodic solution to \eqref{eq:hcln}. Therefore,
\begin{equation*}
  \big\{|u(n)|=e^{\ln |u(n)|}\big\}_{n\in \mathbb{Z}}
\end{equation*}
is an almost periodic solution to \eqref{eq:hc+}. By \Cref{lem:hcsgn-s}, $\{\sgn V(n, 0)\}_{n\in \mathbb{Z}}$ is a periodic solution to \eqref{eq:hcsgn}. A direct calculation shows that
\begin{align*}
u(n)&=V(n, 0)u(0)\\
&=|V(n, 0)|\cdot |u(0)|\cdot \sgn[V(n, 0)u(0)]\\
&=|u(n)|\cdot \sgn V(n, 0)\cdot \sgn u(0)
\end{align*}
for all $n\in \mathbb{Z}$. Thus $u\in AP(\mathbb{Z}, \mathbb{R})$.
\end{proof}

\begin{remark}\label{re:hc-ap}
Under the conditions of \Cref{th:hc-ap}, \eqref{eq:hc} has a nontrivial solution $u$ with $\inf_{n\in \mathbb{Z}} |u(n)|=0$ if and only if $\{\ln |u(n)|\}_{n\in \mathbb{Z}}$ is a solution to \eqref{eq:hcln} such that $\inf_{n\in \mathbb{Z}} \ln |u(n)|=-\infty$, which may be true in many situations.
\end{remark}

\section{The factorization theorem}\label{sec:factor}

In this section, we prove the main result on the factorization of p.c.a.p. functions into Bohr almost periodic functions and sequences, and show the almost periodicity of the functions in \eqref{eq:change2}.

\begin{proof}[Proof of \Cref{th:factor1}]
If $a(n):=h(\tau_n^+)/h(\tau_n)$ for $n\in \mathbb{Z}$, then $a\in AP(\mathbb{Z}, \mathbb{R})$ and $\inf_{n\in \mathbb{Z}} |a(n)|>0$ by (iii) of \Cref{lem:pcap}. \Cref{th:hc-ap} implies that the bounded solution $v$ with $\inf_{n\in \mathbb{Z}} |v(n)|>0$ to \eqref{eq:factor-c} is almost periodic. Define
\begin{align*}
&w(n)=\frac{1}{v(n)}, \quad n\in \mathbb{Z},\\
&f(t)=h(t)w(n), \quad \tau_n<t\leq \tau_{n+1}, n\in \mathbb{Z}.
\end{align*}
Then $w\in AP(\mathbb{Z}, \mathbb{R})$. We shall show $f\in AP(\mathbb{R}, \mathbb{R})$. The proof of this statement is divided into three steps.

1. We prove that $f$ is continuous on $\mathbb{R}$. Since $v$ is a solution to \eqref{eq:factor-c}, a direct calculation shows that
\begin{equation*}
  f(\tau_n^+)=\frac{h(\tau_n^+)}{v(n)}=\frac{h(\tau_n)}{v(n-1)}=f(\tau_n)
\end{equation*}
for all $n\in \mathbb{Z}$.

2. We prove that $f$ is uniformly continuous on $\mathbb{R}$, a property which will be used later. Given $\epsilon'>0$, let $\theta=\inf_{j\in \mathbb{Z}} \tau_j^1$ and $\delta$, $0<\delta<\theta$ be a number with $|h(s)-h(t)|<\epsilon'/\|w\|$ whenever $s$, $t\in (\tau_j, \tau_{j+1}]$ for some $j\in \mathbb{Z}$ and $|s-t|<\delta$ by (ii) of \Cref{df:pcap}. We shall consider two cases of the points in $\mathbb{R}$.

Case 1. $s$, $t\in (\tau_m, \tau_{m+1}]$ for some $m\in \mathbb{Z}$ and $|s-t|<\delta$. It is easy to see that
\begin{equation*}
|f(s)-f(t)|=|[h(s)-h(t)]w(m)|<\epsilon'.
\end{equation*}

Case 2. $s\leq \tau_m<t$ for some $m\in \mathbb{Z}$ and $|s-t|<\delta$. Since $\delta<\theta$, $s>\tau_{m-1}$ and $t<\tau_{m+1}$. From the continuity of $f$ at $\tau_m$ and Case 1 it follows that
\begin{equation*}
|f(s)-f(t)|\leq |f(s)-f(\tau_m)|+|f(\tau_m^+)-f(t)|<2\epsilon'.
\end{equation*}

In all, $|f(s)-f(t)|<2\epsilon'$ whenever $|s-t|<\delta$.

3. Let $\epsilon'$ and $\delta$ be the same as in Step 2. We prove that $f\in AP(\mathbb{R}, \mathbb{R})$ by showing the set $T(f, 5\epsilon')$ to be relatively dense. Choose $\epsilon$ so small that
\begin{equation*}
  0<\epsilon<\min\Big\{\frac{\epsilon'}{\|w\|+\|h\|}, \frac{\delta}{2}\Big\},
\end{equation*}
where $\|h\|:=\sup_{t\in \mathbb{R}} |h(t)|<\infty$. Consider the following inequalities
\begin{align}
&|h(t+r)-h(t)|<\epsilon, \quad t\in \mathbb{R}, |t-\tau_j|>\epsilon, j\in \mathbb{Z},\label{eq:factor-4}\\
&|w(n+q)-w(n)|<\epsilon, \quad n\in \mathbb{Z},\label{eq:factor-5}\\
&|h(\tau_{n+q})-h(\tau_n)|<\epsilon, \quad n\in \mathbb{Z}\label{eq:factor-6},\\
&|\tau_j^q-r|<\epsilon, \quad j\in \mathbb{Z}.\label{eq:factor-3}
\end{align}
Using the method of common almost periods as in Lemma 35 in \cite[p. 208]{SP95}, the following two sets
\begin{align*}
&\Gamma_1=\{r\in \mathbb{R}; \text{ there exists } q\in \mathbb{Z} \text{ such that } (r, q) \text{ satisfies \eqref{eq:factor-4}-\eqref{eq:factor-3}}\},\\
&Q_1=\{q\in \mathbb{Z}; \text{ there exists } r\in \mathbb{R} \text{ such that } (r, q) \text{ satisfies \eqref{eq:factor-4}-\eqref{eq:factor-3}}\}
\end{align*}
are relatively dense. Let $(r, q)\in \Gamma\times Q$ satisfy \eqref{eq:factor-4}-\eqref{eq:factor-3}. We shall study two possible cases.

Case 1. $\tau_m+\epsilon<t<\tau_{m+1}-\epsilon$ for some $m\in \mathbb{Z}$. From \eqref{eq:factor-3} it follows that
\begin{align*}
&\tau_{m+q}-\tau_m-\epsilon<r<\tau_{m+1+q}-\tau_{m+1}+\epsilon,\\
&\tau_{m+q}<t+r<\tau_{m+q+1}.
\end{align*}
Consequently,
\begin{align*}
|f(t+r)-f(t)|&=|h(t+r)w(m+q)-h(t)w(m)|\\
&\leq |[h(t+r)-h(t)]w(m+q)|+|h(t)[w(m+q)-w(m)]|\\
&<(\|w\|+\|h\|)\cdot \epsilon<5\epsilon'.
\end{align*}

Case 2. $\tau_m-\epsilon\leq t\leq \tau_m+\epsilon$ for some $m\in \mathbb{Z}$. From \eqref{eq:factor-3} it follows that
\begin{align*}
&\tau_{m+q}-\tau_m-\epsilon<r<\tau_{m+q}-\tau_m+\epsilon,\\
&\tau_{m+q}-2\epsilon<t+r<\tau_{m+q}+2\epsilon.
\end{align*}
Hence $|t+r-\tau_{m+q}|<2\epsilon<\delta$. Therefore,
\begin{align*}
|f(t+r)-f(t)|&\leq |f(t+r)-f(\tau_{m+q})|+|f(\tau_{m+q})-f(\tau_m)|+|f(\tau_m)-f(t)|\\
&<2\epsilon'+|h(\tau_{m+q})w(m+q-1)-h(\tau_m)w(m-1)|+2\epsilon'\\
&<4\epsilon'+|[h(\tau_{m+q})-h(\tau_m)]w(m+q-1)|\\
&+|h(\tau_m)[w(m+q-1)-w(m-1)]|\\
&<4\epsilon'+(\|w\|+\|h\|)\cdot \epsilon<5\epsilon'
\end{align*}
by the uniform continuity of $f$ and \eqref{eq:factor-5}, \eqref{eq:factor-6}.

Summing up, $T(f, 5\epsilon')$ contains the relatively dense set $\Gamma$ and $f\in AP(\mathbb{R}, \mathbb{R})$.

At last, the uniqueness of $f$ and $v$ up to a nonzero multiplicative constant follows from \Cref{lem:hc-s}.
\end{proof}

\begin{remark}\label{re:factor}
From the proof above it follows that \Cref{th:factor1} remains true if \eqref{eq:factor-h} is replaced by
\begin{equation}\label{eq:factor-h2}
  \Big\{\frac{h(\tau_n^+)}{h(\tau_n)}\Big\}_{n\in \mathbb{Z}}\in AP(\mathbb{Z}, \mathbb{R}), \quad \inf_{n\in \mathbb{Z}} \Big|\frac{h(\tau_n^+)}{h(\tau_n)}\Big|>0.
\end{equation}
Note that \eqref{eq:factor-h} is proposed to generate the sequence in \eqref{eq:factor-h2}. This remark will be referenced in the applications of impulsive differential equations.
\end{remark}

\begin{proof}[Proof of \Cref{th:factor2}]
Let $h$ be given by \eqref{eq:factor}. The proof is divided into three steps.

1. We prove that $h$ satisfies (ii) of \Cref{df:pcap}. Since $f$ is uniformly continuous on $\mathbb{R}$, for every $\epsilon>0$ there is a $\delta=\delta(\epsilon)>0$ such that $|f(s)-f(t)|<\epsilon$ for all $s$, $t\in \mathbb{R}$ satisfying $|s-t|<\delta$. A straightforward computation shows that
\begin{equation*}
|h(s)-h(t)|=|f(s)u(j)-f(t)u(j)|\leq \|u\|\cdot \epsilon
\end{equation*}
whenever $s$, $t\in (\tau_j, \tau_{j+1}]$ for some $j\in \mathbb{Z}$ and $|s-t|<\delta$.

2. We prove that $h$ satisfies (iii) of \Cref{df:pcap} by showing the set $T(h, \epsilon')$ to be relatively dense, where $\epsilon'=(\|u\|+\|f\|)\cdot \epsilon$. Consider \eqref{eq:factor-3} the following inequalities in $(r, q)\in \mathbb{R}\times \mathbb{Z}$,
\begin{align}
&|f(t+r)-f(t)|<\epsilon, \quad t\in \mathbb{R},\label{eq:factor-1}\\
&|u(n+q)-u(n)|<\epsilon, \quad n\in \mathbb{Z},\label{eq:factor-2}
\end{align}
Using the method of common almost periods as in Lemma 35 in \cite[p. 208]{SP95}, the following two sets
\begin{align*}
&\Gamma=\{r\in \mathbb{R}; \text{ there exists } q\in \mathbb{Z} \text{ such that } (r, q) \text{ satisfies \eqref{eq:factor-3}, \eqref{eq:factor-1} and \eqref{eq:factor-2}}\},\\
&Q=\{q\in \mathbb{Z}; \text{ there exists } r\in \mathbb{R} \text{ such that } (r, q) \text{ satisfies \eqref{eq:factor-3}, \eqref{eq:factor-1} and \eqref{eq:factor-2}}\}
\end{align*}
are relatively dense. Let $(r, q)\in \Gamma\times Q$ satisfy \eqref{eq:factor-3}, \eqref{eq:factor-1} and \eqref{eq:factor-2}. If $\tau_m+\epsilon<t<\tau_{m+1}-\epsilon$ for some $m\in \mathbb{Z}$, from \eqref{eq:factor-3} it follows that
\begin{align*}
&\tau_{m+q}-\tau_m-\epsilon<r<\tau_{m+1+q}-\tau_{m+1}+\epsilon,\\
&\tau_{m+q}<t+r<\tau_{m+q+1}.
\end{align*}
A direct calculation shows that
\begin{align*}
|h(t+r)-h(t)|&=|f(t+r)u(m+q)-f(t)u(m)|\\
&\leq|[f(t+r)-f(t)]u(m+q)|+|f(t)[u(m+q)-u(m)]|\\
&<(\|u\|+\|f\|)\cdot \epsilon=\epsilon'.
\end{align*}
Hence $T(h, \epsilon')$ contains the relatively dense set $\Gamma$ and $h$ is p.c.a.p.

3. In this case, \eqref{eq:factor-c} takes the form of
\begin{equation*}
  x(n)=\frac{f(\tau_n)u(n)}{f(\tau_n)u(n-1)}\cdot x(n-1)=\frac{u(n)}{u(n-1)}\cdot x(n-1).
\end{equation*}
By \eqref{eq:factor-h}, it is obvious that $u$ is a nontrivial almost periodic solution to \eqref{eq:factor-c}. So \Cref{th:hc-ap} implies $\inf_{n\in \mathbb{Z}} |u(n)|>0$.
\end{proof}

At last, let us show the almost periodicity of the functions in \eqref{eq:change2}.
\begin{proof}[Proof of \Cref{th:change}]
It is easy to check that $(\phi_1, \ldots, \phi_d)$ is a solution to \eqref{eq:IDE2} if and only if $(\phi_1/\psi_1, \ldots, \phi_d/\psi_d)$ is a solution to \eqref{eq:DE2}. Next we prove the almost periodicity.

Let $(\phi_1, \ldots, \phi_d)$ be a p.c.a.p. solution to \eqref{eq:IDE2} such that $\phi_i(\tau_n)\neq 0$ for all $n\in \mathbb{Z}$ and $i=1$, $\ldots$, $d$. From
\begin{equation*}
  \frac{\phi_i(\tau_n^+)}{\phi_i(\tau_n)}=\frac{\psi_i(\tau_n^+)}{\psi_i(\tau_n)}=1+b_i(n), \quad n\in \mathbb{Z}.
\end{equation*}
it follows that \eqref{eq:factor-c} takes the form of \eqref{eq:CE1} for $i=1$, $\ldots$, $d$. Therefore, the sequences $\{u_i\}_{i=1}^d$ defined by \eqref{eq:seq} are almost periodic by (A2) and \Cref{th:hc-ap}, and the functions $\{\psi_i\}_{i=1}^d$ given by \eqref{eq:fct} are p.c.a.p. by \Cref{th:factor2} with $f=1$. Since $u_i$ is a nontrivial solution to \eqref{eq:CE1} and
\begin{equation}\label{eq:change3}
  \phi_i(t)=\frac{\phi_i(t)}{\psi_i(t)}\cdot u(n), \quad \tau_n<t\leq \tau_{n+1}, n\in \mathbb{Z},
\end{equation}
\Cref{re:factor} implies that $\phi_i/\psi_i$ is the unique Bohr almost periodic function such that \eqref{eq:factor} holds for $\phi_i$, $\phi_i/\psi_i$ and $u_i$.

The second conclusion follows from \eqref{eq:change3} and \Cref{th:factor2}.
\end{proof}

\section*{Acknowledgments}

The author R. Yuan is supported by The National Nature Science Foundation of China (Grant No. 11771044).

\end{document}